\documentclass[11pt,english]{smfart}

\usepackage{amsmath,leftidx,amsthm,sfheaders}
\usepackage[all]{xy}
\usepackage{xspace,smfthm}
\usepackage[psamsfonts]{amssymb}
\usepackage[latin1]{inputenc}
\usepackage{graphicx,color}
\usepackage{hyperref,fancyhdr}

\newtheorem{theorem}{\bf Theorem}[section]
\newtheorem{proposition}[theorem]{\bf Proposition}
\newtheorem{definition}[theorem]{\bf Definition}
\newtheorem{Theorem}{\bf Theorem}

\newtheorem{lemma}[theorem]{\bf Lemma}
\newtheorem{corollary}[theorem]{\bf Corollary}

\theoremstyle{definition}

\newtheorem*{remark}{\bf Remark}

\def\C{{\mathbb C}}
\def\N{{\mathbb N}}

\def\P{\mathbb{P}}

\def\bif{\textup{bif}}

\def\Mand{\mathbf{M}}

\def\and{{\quad\text{and}\quad}}

\title{Distribution of points with prescribed derivative in polynomial dynamics}
\author{Thomas Gauthier $\&$ Gabriel Vigny}
\address{LAMFA, Universit\'e de Picardie Jules Verne, 33 rue Saint-Leu, 80039 AMIENS Cedex 1, FRANCE}
\email{thomas.gauthier@u-picardie.fr, gabriel.vigny@u-picardie.fr}
\thanks{Both authors are partially supported by the ANR grant Lambda ANR-13-BS01-0002.}\thanks{The first author is partially supported by a PEPS ``Jeune-s Chercheur-e-s'' Grant.}
\begin{document}
\begin{abstract}
In analogy to the equidistribution of preimages of a prescribed point by the iterates of a polynomial map $f$ in $\C$ towards the equilibrium measure, we show here the equidistribution of points $z$ for which $(f^n)'(z)=a$ for suitable $a$ towards the equilibrium measure. We then give a similar statement in the space of degree $d$  polynomials for the equidistribution of parameters for which the $n$-derivative at a given critical value has a prescribed derivative towards the activity current of the corresponding critical point.
\end{abstract}

\maketitle

\section{Introduction}
 
In the first part of the article, we are interested in the equidistribution of points with prescribed derivative for a polynomial map $f:\C\rightarrow\C$ of degree $d\geq 2$ in $\C$. Recall for that  that the \emph{Green function} of $f$ is 
\[g_{f}(z):=\lim_{n\rightarrow\infty}d^{-n}\log\max\left\{1,|f^n(z)|\right\}~, \ z\in\C~.\]
The Julia set $J_f$ of $f$ is $J_f:=\partial\{g_f=0\}$. The probability measure $\mu_f:=dd^c_zg_f$ is known as the \emph{equilibrium measure} of $f$. It is the unique  measure of maximal entropy $\log d$ of $f$ and its support is the Julia set of $f$. This measure describes many equidistribution phenomena, notably the following: there exists a set $E$ containing at most one point such that for all $a\in \C \backslash E$, the measure equidistributed on the preimages of $a$ converges to $\mu_f$:
\[\lim_n \frac{1}{d^n} \sum_{ f^n(z)= a} \delta_z = \mu_f \] 
in the sense of measure (we take into account the multiplicity in the sum). This result is due to Brolin~\cite{Brolin} and has been extended to the case of rational maps by Ljubich~\cite{ljubich} and independently by Freire, Lopes and Ma{\~n}{\'e}~\cite{FLM}, a quantified statement has been established in  \cite{DrasinOkuyama}. See \cite{FornaessSibony, briendduval2, DinhSibony3} for similar results in higher dimensions.

We want to understand a similar statement but for derivatives. Let us give some motivations for that :
\begin{itemize}
\item the solutions of $f^n(z)=z$ are exactly the $n$-periodic points which are known to equidistribute towards the equilibrium measure (see \cite{ljubich}). Then, one can ask what is exactly the multiplicity of the solutions and for that one need to solve   $f^n(z)=z$ and  $(f^n)'(z)=1$. Then, if the solutions of $(f^n)'(z)=1$ were far from the Julia set, one could conclude that the solution of $f^n(z)=z$ are mostly simple. In fact, Theorem~\ref{equi_poly} below says that the solutions of $(f^n)'(z)=1$ tends to accumulate on the Julia set $J_f$ of $f$, making that approach ineffective. Similarly, assume that one wants to compute the Lyapunov exponent of $f$ by computing the derivative $f^n(z)$ at some generic point $z$ in the support of $\mu_f$,  Theorem~\ref{equi_poly} shows that a small error in the selection of $z$ can give any possible result for the derivative!  
\item by the chain rule, the map $(z,n) \mapsto (f^n)'(z)$ defines a cocycle, the article thus addresses the question of the equidistribution of the preimages of a prescribed target by a cocyle in the simplest case. 
\item more generally, we think of polynomial maps of $\C$ as a test case, the questions raised in the article can be asked in a lot of situations (e.g. H\'enon mappings, rational maps in higher dimension). We will study the equidistribution towards activity currents in the second part of the article.
\end{itemize} 

We now state our results. For $\lambda \in \C$, we denote by  $\nu_n^\lambda$ the following probability measure:
\[\nu_n^\lambda :=\frac{1}{d^n-1} \sum_{ (f^n)'(z)= \lambda } \delta_z,\] 
where the sum is taken with multiplicity. Our first result is the following equidistribution statement of $\nu_n^\lambda$ towards the equilibrium measure $\mu_f$ of $f$:
\begin{Theorem}\label{equi_poly}
Let $f:\C\longrightarrow\C$ be a polynomial map of degree $d\geq2$ and let $\nu_n^\lambda$ and $\mu_f$ be the measures defined above. There exists a polar set $E_f\subset\C^*$ such that
\begin{enumerate}
\item For all $\lambda\in\C\setminus E_f$, one has $\nu_n^\lambda \to \mu_f$ in the sense of measures,
\item If $f$ has no siegel disk and no escaping critical points, one has $E_f=\emptyset$,
\item If $f$ is hyperbolic, one also has then $E_f=\emptyset$.
\end{enumerate}
\end{Theorem}

Recall that $f$ is hyperbolic if it is uniformly strictly expanding on its Julia set; or equivalently if the set $\overline{\{f^n(c)\, ; \ n\geq0, \ f'(c)=0\}}$ is disjoint from the Julia set $J_f$ of $f$.

The proof of the first point (see Theorem~\ref{equi_poly1}) is deduced from the following interesting proposition (see subsection~\ref{basics_dsh} for the definition of PB measures):
\begin{proposition}\label{prop_equi_poly}
Let $f:\C\longrightarrow\C$ be a polynomial map of degree $d\geq2$ and let $\nu$ be a PB measure on $\P^1$, then we have the convergence:
\[ \lim_{n\to \infty} \frac{1}{d^n-1} \left((f^n)'\right)^*(\nu) =  \mu_f \]
\end{proposition}

\begin{remark} \normalfont
\begin{enumerate}
\item In particular, for quadratic polynomial maps $f$, the result is true for all $\lambda\neq0$ as long as $f$ does not have a Siegel disk.
\item The case $\lambda=0$ is the equidistribution of the preimages of the critical sets which is known to fail if and only if some critical points is in the exceptional set, i.e. $f$ is affine conjugate to $z^d$.
\item Theorem~\ref{equi_poly} can be, at least partially, extended to rational maps of $\P^1$ (see Remarks~\ref{remark_equidistibution1} and \ref{remark_equidistibution2}). Nevertheless,  Theorem~\ref{equi_poly} is invariant under affine conjugacy whereas it is not under Moebius conjugacy hence we choose to stick to polynomial maps.
\end{enumerate}
\end{remark}
The idea of the proof of the first point of Theorem~\ref{equi_poly} is to study the dynamics of the tangent map $F(z,u)=(f(z),f'(z)\cdot u)$ (in $\C^2$). We show that its Green current is in fact the pull back of $ \mu_f$ by the projection $\pi_1$  on the first coordinate and that $(d^n-1)^{-1}(F^n)^*([u=1])$ converges the Green current (Proposition~\ref{for_u_0}). Then we show the convergence (for $\lambda$ outside a pluripolar set) of the  intersection $(d^n-1)^{-1}(F^n)^*([u=1]) \wedge [u=\lambda]$ towards $\pi_1^*(\mu_f)\wedge [u=\lambda]$ which concludes the proof (Theorem~\ref{equi_poly1}). It should be possible to show that $E_f$ is empty with that approach using the recent theory of Dinh and Sibony of density of currents \cite{DS14}, as we explain in Remark~\ref{remark_equidistibution1}. We then give the proof of the second and third points of Theorem~\ref{equi_poly} using the classical approaches of Brolin and Ljubich, though they seem to generalize only to very specific other cases.

~
 
In the second part of the article, we focus on bifurcation phenomena in parameter spaces of polynomial maps of $\C$ of a given degree $d\geq2$. For $c=(c_1,\ldots,c_{d-2})\in\C^{d-2}$ and $a\in\C$, we let
\[P_{c,a}(z):=\frac{1}{d}z^d+\sum_{j=2}^{d-1}(-1)^{d-j}\frac{\sigma_{d-j}(c)}{j}z^j+a^d~, \ z\in\C~,\]
with $\sigma_k(c)$ the monic homogeneous degree $k$ symmetric function in the $c_i$'s. 
This family is known to be a finite branched cover of the \emph{moduli space} $\mathcal{P}_d$ of critically marked degree $d$ polynomials, i.e. the space of affine conjugacy classes of degree $d$ polynomials with $d-1$ marked critical points (see e.g.~\cite[\S 5]{favredujardin}). The critical points of $P_{c,a}$ are exactly $c_0,\ldots,c_{d-2}$, with the convention $c_0:=0$.

~

Pick $0\leq i\leq d-2$. As it now is classical, we say that a critical point $c_i$ is \emph{passive} at $(c_*,a_*)\in \C^{d-1}$ if there exists a neighborhood $U\subset \C^{d-1}$ of $(c_*,a_*)$ such that the sequence $F^i_n:\C^{d-1}\to\C$ of holomorphic maps defined by
\[F^i_{n}(c,a):=P_{c,a}^n(c_i)\]
is a normal family on $U$. Otherwise, we say that $c_i$ is \emph{active} at $(c_*,a_*)$. The \emph{activity locus} of $c_i$ is the set of parameters $(c,a)\in\C^{d-1}$ such that $c_i$ is active at $(c,a)$.  
We can define an \emph{activity current} $T_i$ to give a measurable sense to the notion of activity.
We denote by $\pi_{d-1}$ (resp. $\pi_1$) the canonical projections on $\mathbb{C}^{d-1}$ (resp. on $\C$).
Let $\mathcal{T}$ be the Green current of the map $f$ (it can be defined as the limit of $d^{-n} (f^n)^*(\pi_1^*(\omega_1))$). Then, by the work~\cite{BB1} of Bassanelli and Berteloot, we have $T_i := (\pi_{d-1})_*(\mathcal{T} \wedge [z= c_i] )$. The invariance of the Green current implies that it can also be defined by intersecting with the graph of the critical value $P_{c,a}(c_i)$ as below:
\begin{equation}\label{def_T_i}
T_i =  \frac{1}{d}(\pi_{d-1})_*\left(\mathcal{T} \wedge \left[z= P_{c,a}(c_i)\right] \right).
\end{equation}
 As proved by Dujardin and Favre~\cite{favredujardin}, the current $T_i$ is exactly supported by the activity locus of $c_i$. Moreover, they prove that the sequence of smooth forms $d^{-n}(F^i_n)^*\omega_{\P^1}$ converges in the weak sense of currents to $T_i$.

The currents $T_i$ and $\sum_i T_i$  are known to equidistribute various phenomena: parameters for which the critical point $c_i$ is preperiodic with a given orbit portrait \cite{favredujardin}, parameters admitting a cycle with a given multiplier \cite{BB2, BB3,multipliers,DistribTbif}, or parameters at which the critical points are sent to some prescribed target \cite{Dujardin2012, distribGV}.

As in the case of polynomial maps, we prove here the following:
\begin{Theorem}\label{equi_bif}
Pick any integer $0\leq i\leq d-2$. Then the following convergences holds in the weak sense of currents on $\C^{d-1}$:
\begin{enumerate}
\item for any probability measure $\nu$ with bounded potential on $\C$,
\[\lim_{n\rightarrow\infty}\frac{1}{d^n}\int_{\C}\left[(P_{c,a}^n)'(P_{c,a}(c_i))=\lambda\right]\mathrm{d}\nu(\lambda)=T_i~,\]
\item there exists a polar set $E_i\subset \C^*$ such that, for any $\lambda\in\C\setminus E_i$,
\[\lim_{n\rightarrow\infty}\frac{1}{d^n}\left[(P_{c,a}^n)'(P_{c,a}(c_i))=\lambda\right]=T_i~.\]
\end{enumerate}
\end{Theorem}

\par Finally, we focus on the quadratic family $p_c(z):=z^2+c$, $c\in\C$. In that very particular context, we can prove a stronger statement, which in particular imply the exceptional set is empty (see Theorem~\ref{equiMand}). Here this means that for any $\lambda\in\C$, the sequence of finite measures $\frac{1}{2^n}[(p_c^n)'(c)=\lambda]$ converges to the harmonic measure of the Mandelbrot set.


\section{In the phase space of a complex polynomial}
The aim of the present section is to prove Theorem~\ref{equi_poly}.
\subsection{A tangent map}
Let $f$ be a polynomial map of $\mathbb{C}$ of degree $d\geq2$. We can write it as $f(z)=\sum_{i=0}^d a_i z^i$. We consider the tangent map $F(z,u)=(f(z),f'(z)\cdot u) $ acting on the tangent bundle $\mathrm{Tan}(\C)$ that we write in the birational model $\mathbb{P}^1 \times \P^1$ as:
\[ F([z:t],[u:v])= \left(\left[ \sum_{i=0}^d a_i z^it^{d-i}: t^d\right] , \left[\left(\sum_{i=0}^d i a_i z^{i-1}t^{d-i}\right)u: t^d v \right] \right)                   . \] 
If $c_1, \dots , c_{d-1}$ are the critical points of $f$ in $\C$ (counted with multiplicity), then the indeterminacy points of $F$ are the points $I_j= (c_j,[1:0])$ and the points $I_\infty:=([1:0],[0:1])$. We denote by $I(F)$ the union of those points.
Observe that the (invariant) set $\{[1:0]\}\times \P^1\backslash I_\infty$ is sent to $I':=([1:0],[1:0])$ by $F$. Finally, observe that the point $I'$ is attracting.

~

 Let $\omega_i$ be the pull back of the Fubini-Study form $\omega_{\P_1}$ on $\P^1$ by the canonical projection to the $i$-th factor of  $\pi_i: \P^1\times \P^1 \to \P^1$. Let $\{\omega_i\}$ denote the class of $\omega_i$ in $H^{1,1}( \P^1\times \P^1)$. 
\begin{lemma}
The map $F$ is algebraically stable i.e. $(F^*)^n=(F^n)^*$ in $H^{1,1}( \P^1\times \P^1)$. The action of $F^*$ is given in the $(\{\omega_1\}, \{\omega_2\})$ basis by the matrix:
\[\begin{pmatrix} d & d-1 \\
                    0  & 1 
\end{pmatrix}.\]
 Its topological degree $d_t$ is $d$, in particular, it is not cohomologically hyperbolic. 
\end{lemma}
\begin{proof}
The obstruction to the algebraic stability is the fact that some hypersurface (i.e. some curve here) is sent to some indeterminacy point by some iterate $F^n$ of $F$ (\cite{Sibony}). As $f$ is holomorphic, the only possibility is that some $\{p\}\times \P^1$ is sent to an indeterminacy set by $F^n$. By the chain rule, $F^n$ is the extension of $F^n(z,u)=(f^n(z),(f^n)'(z)\cdot u) $ on the tangent bundle $\mathrm{Tan}(\C)$ to its compactification $\mathbb{P}^1 \times \P^1$. Let $p \in \C$, observe that  $\{p\}\times\{[0:1]\}\subset \{p\}\times \P^1$ is sent to $\{f(p)\}\times\{[0:1]\}$ whether $p$ is critical or not. It follows from this that no hypersurface $\{p\}\times \P^1$ is sent to an indeterminacy point by $F^n$ for $p\in \C$. On the other hand, the point $I'$ is sent to itself by $F$ so $F^n([1:0]\times \P^1) \neq I_\infty$. So $F$ is algebraically stable. The rest of the proof follows. Observe that the first dynamical degree is by definition the spectral radius of the matrix
\[\begin{pmatrix} d & d-1 \\ 0  & 1 \end{pmatrix}\]
hence it is equal to $d=d_t$, by definition, it implies that $F$ is not cohomologically hyperbolic.
\end{proof}
The study of cohomologically hyperbolic maps is very well developed (see e.g. \cite{Guedj}), no general theory exists for not cohomologically hyperbolic. 

\subsection{Basics on dsh functions and PB measures}\label{basics_dsh}

Let us recall some facts on dsh functions (see e.g.~\cite{DinhSibonyregular}). Recall that a probability
$\nu$ in $(\P^1)^2$ has bounded quasi-potentials (or is PB) if $\nu$ admits a negative quasi-potential 
$U$ ($dd^c U+ \Omega= \nu$ where $\Omega$ is some smooth probability measure) such that $|\langle U, S\rangle| \leq C$ for any positive smooth form $S$ 
of bidegree $(1,1)$ and mass $1$. Such notion of quasi-potential can be extended to any positive 
closed current of bidegree $(1,1)$ and mass $1$  with the same bound $|\langle U, S\rangle| \leq C$. In particular, any smooth measure is PB.

\begin{definition}
We say that a function $\varphi$ on $(\P^1)^{2}$ is \emph{dsh} if, outside a pluripolar set, it can be written as a difference of quasi-psh functions.
\end{definition}
 
For example, if $\varphi \in \mathcal{C}^2$, then it is dsh.

Let $DSH\left(\left(\P^1\right)^{2}\right)$ be the space of such functions on $(\P^1)^2$. For any $\varphi\in DSH\left(\left(\P^1\right)^{2}\right)$, we write $dd^c \varphi=T^+-T^-$ where $T^\pm$ are positive closed currents of bidegree $(1,1)$.
Let $\nu$ be a PB measure on $(\P^1)^{2}$. The following defines a norm on the space $DSH\left(\left(\P^1\right)^{2}\right)$:
\[\|\varphi\|_{\nu}:=  \|\varphi\|_{L^1(\nu)} +\inf \|T^\pm\|~,\]
where the infimum is taken on all the decompositions $dd^c \varphi=T^+-T^-$ as above. It turns out that taking another PB measure $\nu'$ gives an equivalent norm on $DSH\left(\left(\P^1\right)^{2}\right)$ (see e.g. \cite[p. 283]{dinhsibony2}).

Finally, recall that a Borel set that is of measure $0$ for all the PB measures is in fact pluripolar.

\subsection{A first convergence property}
Recall that the \emph{Green measure} $\mu_f$ of the map $f$ can be defined as the limit
\[\mu_f:=\lim_n d^{-n} (f^n)^*(\omega_{\P_1})~.\]
It then follows from the fact that $F$ is a skew-product that the sequence of positive closed current $d^{-n} (F^n)^*(\omega_1)$ converges to the \emph{Green current} $T_F$ of $F$ which is a positive closed current of mass $1$ and that $T_F= \pi_1^*(\mu_f)$.

\begin{lemma}\label{lm:cvGreen}
The sequence of positive closed currents $d^{-n} (F^n)^*(\omega_2)$ converges to the Green current $T_F$ of $F$.
\end{lemma}

\begin{proof}
Write $\omega_2= \omega_1 + \omega_2-\omega_1$. Since $d^{-n} (F^n)^*(\omega_1)$ converges to $T_F$, all there is left to prove is that $d^{-n} (F^n)^*(\omega_2- \omega_1)$ converges to $0$ in the sense of currents. Observe that $F^*(\omega_2- \omega_1)$ is cohomologous to $\omega_2- \omega_1$ hence we can write 
\[F^*(\omega_2- \omega_1)= \omega_2- \omega_1+dd^c \varphi\]
where $\varphi$ is smooth outside $I(F)$ and is a dsh function (in fact $\varphi$ is quasi-psh since $F^*(\omega_1)$ is a smooth form).

Let $W$ be a small neighborhood of $I'$ such that $F(W)\subset W$ and where $\varphi$ is uniformly bounded. Let $\nu_W$ be a smooth (hence PB) probability measure with support  in $W$. Then, a straight-forward induction gives:
\[d^{-n} (F^n)^*(\omega_2- \omega_1)= d^{-n}(\omega_2- \omega_1)+dd^c  \left(d^{-n}\sum_{k=0}^{n-1}  \varphi \circ F^k \right).\]
The sequence of function $\varphi_n:=d^{-n}\sum_{k=0}^{n-1}  \varphi \circ F^k$ is then a sequence of dsh functions. Furthermore, as $F(W)\subset W$, we see that $\|\varphi_n\|_{\infty, W}\leq nd^{-n} \|\varphi\|_{\infty, W}\leq C$ where $C$ is a constant that does not depend on $n$.  In particular, $\|\varphi_n\|_{\nu_W}$ is uniformly bounded. As all the $DSH$-norms are equivalent, we deduce that $(\varphi_n)$ is bounded in $L^1$ for the standard Fubini Study measure in $(\P^1)^{2}$. 

On the other hand, the sequence $d^{-n} (F^n)^*(\omega_2)$ is bounded in mass hence we can extract a converging subsequence. Its limit is a positive closed current cohomologuous to $\omega_1$ hence it can be written as $\pi^*(\nu')$ where $\nu'$ is some probability measure in $\P^1$. Extracting again, we can assume that $(\varphi_n)$  converges in $L^1$ (it is bounded in DSH). Its limit $V$ satisfies $dd^c V= (\pi_1)^*(\nu'-\mu_f)$ hence it is constant on each fiber of $\pi_1$. In other words, $V=\pi_1^*(v)$ for some dsh function $v$ on $\P^1$. Take $z \in \C$ in  the interior of the filled Julia set. Then, the sequences $(f^n(z))$ and $(f^n)'(z)$ are equicontinuous near $z$, hence bounded. In particular, for any $u\in \C$, we have that the sequence $(f^n(z), (f^n)'(z)\cdot u)$ stays in a some compact subset of $\C \times \C$ (this compact set can even be chosen to be uniform in a neighborhood of $(z,u)$). In particular, the sequence $\varphi_n(z,u)$ converges to $0$. Similarly, for $[z:t]$ in the basin of attraction of $\infty$ and any $[u:v]$ with $u\neq 0$, we have that the sequence $([f^n(z:t): t^{d^n}], [(f^n)'(z:t)u: t^{d^n-1}v] )$ converges to $([1:0],[1:0])$. Hence, $\varphi_n ([z:t],[u:v])$ converges to $0$ for such $([z:t],[u:v])$.

It follows that the function $v$ is equal to $0$ in the Fatou set. Now, the Julia set has empty interior and a dsh function that is $0$ outside such set is identically $0$ by pluri-fine continuity. It follows that $\nu'=\mu_f$ which ends the proof.
\end{proof}
Any PB probability measure $\nu$ in $\P^1$ can be written as $\omega_{\P^1}+ dd^c \eta$ where $\eta$ is bounded in $\P^1$ so we have the following:

\begin{corollary}\label{cor:PB}
For any PB probability measure $\nu$ in $\P^1$, the sequence of currents $d^{-n} (F^n)^*(\pi_2^*(\nu))$ converges to $T_F$ in the sense of currents.
\end{corollary}

\subsection{The convergence Theorem}

For $u_0 \in \C$, we denote by $[u=u_0]$ the current of integration on the line $\P^1\times \{u_0\}$ of $\P^1\times\P^1$.
\begin{proposition}\label{for_u_0}
For any $u_0 \in \C^*$, the sequence of currents $d^{-n} (F^n)^*[u=u_0]$ converges to $T_F$ in the sense of currents.
\end{proposition}
\begin{proof} By the above Corollary~\ref{cor:PB}, one deduces that outside a polar set of $u_0 \in \C$, $d^{-n} (F^n)^*([u=u_0])$ converges to $T_F$ in the sense of currents. Choose such a $u_0 \neq 0$. For $\lambda \in \C^*$, consider the map $D_\lambda: (z,u) \mapsto (z, \lambda u)$. Observe that 
$$d^{-n} (F^n)^* ( [u=\frac {1}{\lambda} u_0])= D_\lambda^*(d^{-n} (F^n)^* ( [u=u_0])).$$ 
Now, $D_\lambda^*(d^{-n} (F^n)^* ( [u=u_0]))$ converges to $D_\lambda^*(T_F)$ in the sense of currents by assumption and continuity of $D_\lambda^*$. Since $T_F= \pi_1^*(\nu)$, we have $D_\lambda^*(T_F)=T_F$. The result follows since taking a $u_1 \in \C^*$ arbitrary, we can write $ \lambda u_1 = u_0$ for a suitable $\lambda$.
\end{proof}

We can now prove the first part of Theorem~\ref{equi_poly}:
\begin{theorem}\label{equi_poly1}
Let $\nu$ be a probability PB measure in $\C$, then 
\[\lim_{n\to \infty} \frac{1}{d^n}((f^n)')^*\nu = \mu_f~,\]
weakly on $\C$. In particular, outside a polar set of $\lambda \in \C$, we have $\lim\limits_{n\rightarrow\infty}\nu_n^\lambda  = \mu_f$.
\end{theorem}
\begin{proof} 
Let $u_0 \in \C^*$, then observe that $(F^n)^*([u=u_0]) \wedge [u=1] = (F^n)^*([u=1]) \wedge [u=u_0^{-1}]$. Let $\mathrm{inv}:\P^1\to \P^1$ be the rational map $z \mapsto  z^{-1}$. In other words,  
\begin{equation}\label{inv}
  (F^n)^*([u=u_0]) \wedge [u=1] = (F^n)^*([u=1]) \wedge  \pi_2^* (\mathrm{inv}^*([u=u_0])). 
  \end{equation}
By Fubini, we have for $\nu$ 
\[  (F^n)^*(\pi_2^*(\nu)) \wedge [u=1] = (F^n)^*([u=1]) \wedge \pi_2^*( \mathrm{inv}^*(\nu)). \]   
As remarked earlier, proving the statement for one specific PB measure $\nu$ implies the result for all PB measures. In particular, we can take $\nu= \omega_{\P^1}$ .  
By the above Proposition~\ref{for_u_0}, $d^{-n} (F^n)^*([u=1]) \to T_F$ in the sense of currents, hence  $d^{-n} (F^n)^*([u=u_0])\wedge \pi_2^*( \mathrm{inv}^*(\omega_{\P^1} ))\to T_F\wedge \pi_2^*( \mathrm{inv}^*(\omega_{\P^1}))$ since $\pi_2^*( \mathrm{inv}^*(\omega_{\P^1} ))$ has continuous potentials. By continuity of $(\pi_1)_*$, we deduce that 
\[(\pi_1)_*\left( d^{-n} (F^n)^*([u=u_0])\wedge \pi_2^*( \mathrm{inv}^*(\nu)) \right) \to (\pi_1)_*\left( T_F\wedge \pi_2^*( \mathrm{inv}^*(\nu)) \right)= \mu_f,\]
in the sense of currents. Using \eqref{inv} gives the first part of the theorem.

The second part follows since the set of $\lambda$ for which
\[\lim_{n\to \infty} \frac{1}{d^n} \sum_{ (f^n)'(z)= \lambda } \delta_z  = \mu _f\] 
does not hold has zero measure for all the PB measure of $\C$. Such a set is polar.
\end{proof}

\begin{remark}\label{remark_equidistibution1} \normalfont \begin{enumerate}
\item Such a proof has the advantage to be both intuitive and geometric. It may easily be adapted to rational maps of $\P^1$ (some new arguments are needed for maps whose Julia set is the whole $\P^1$) or, as we will see in the next section, to parameter spaces of any dimension. 
\item Nevertheless, it does not give the equidistribution for all $\lambda$ but $0$ without further arguments. This should be possible using the recent theory of Dinh and Sibony of density of currents \cite{DS14}. Indeed, let $G(\mathrm{Tan}(\C),1)$ be the Grassmanian of $1$-plan in $\mathrm{Tan}(\C)$ and let us work in the birational model $(\P^1)^3$. We consider the coordinates $([z:t],[u:v],[a:b])$. For an hypersurface $A$ in $(\P^1)^2$, we let $\widehat{A}$ denote the incidence variety associated to $A$ in $(\P^1)^3$. In particular (working in the chart $\C^2$), for $ (F^n)^*\{u=\lambda\} $, we have the cartesian equation of $\widehat{ (F^n)^*\{u=\lambda\}}$:
\[(f^n)'(z) u = \lambda \  \mathrm{and}\  (f^n)''(z) u a + (f^n)'(z)b = 0 .  \]  
 Then observe that $[\widehat{ (F^n)^*\{u= \lambda\}} ]$ is cohomologous (or at least bounded in cohomology) to:
 $$ 2(d^n-1)\omega_1\wedge \omega_2 + (d^n-1)\omega_1\wedge \omega_3 +\omega_2\wedge \omega_3   $$
 (for the coefficient before $\omega_1\wedge \omega_2$, choose a generic $[a_0:b_0]$ and count the number of solutions of $(f^n)'(z) u =\lambda , \  (f^n)''(z) u a_0 + (f^n)'(z)b_0 = 0$, for that, replace $u$ by  $\lambda((f^n)'(z))^{-1}$ in the second equation). In particular, consider the sequence of positive closed currents (of bidegree $(2,2)$)  :
 $$\frac{1}{d^n}  \left[\widehat{ (F^n)^*\{u=\lambda\}}\right].$$
 We can extract a subsequence which converges to a limit $\widehat{T}$. 
 Then, proceeding as in \cite{DS15}, the result would follow for $\lambda$ provided that:
\[\widehat{T}\curlywedge  \widehat{ [u = 1] } = 0 \]
  where $\widehat{ [u = 1] }$ is the line given by  $u=1$, $b=0$ in $(\P^1)^3$. We were unable to prove this though heuristc arguments show that   $\widehat{T}$ should be $ T_F \wedge [a=0] + 2 \pi_1^*(\mu_f)\wedge [u=0]$. 
\end{enumerate}
\end{remark}

\subsection{Complement 1: Brolin's approach}\label{sec:brolin}
In the rest of the section, we want to present two distinct alternative proofs, giving the second and third items of Theorem~\ref{equi_poly}. The first one is an adaptation of the classical proof of Brolin concerning the distribution of preimages. The second one is an adaptation of the proof of Lyubich concerning the distribution of preimages, which consists in building sufficiently many inverse branches.

~

In this subsection, we assume that $f$ has no Siegel disk. Up to conjugating by a linear map, we can assume that $f$ is unitary. In order to simplify the arguments, we shall also assume that no critical points are in the basin of attraction of infinity (this assumption is only technical). Let $\lambda \in \C^*$, recall that we denoted: 
\[\nu_n^\lambda := \frac{1}{d^n-1} \sum_{(f^n)'(z)=\lambda} \delta_z \]
counting the multiplicity in the sum. We shall prove in this section:
\begin{theorem}\label{equi_poly2}
Let $f$ be a degree $d\geq 2$ polynomial satisfying the above properties and pick any $\lambda \in \C^*$. Then the sequence $(\nu_n^\lambda)_n$ converges towards $\mu_f$ in the sense of measures. 
\end{theorem}
Observe that $(\nu_n^\lambda)$ is a sequence of probability measures so we can extract a converging subsequence towards a limit $\mu'$. We shall show that $\mu'=\mu_f$.
\begin{lemma}
The support of $\mu'$ is contained in the Julia set $J_f$ of $f$.
\end{lemma}
\begin{proof}
Let $\varepsilon >0$ and consider $J_\varepsilon$ an $\varepsilon$-neighborhood of $J_f$. Let $U$ be a bounded Fatou component of $f$. Then in $U \setminus J_\varepsilon$, the sequence $(f^n)$ is normal and converges uniformly towards a constant, since $f$ has no Siegel disk. So the sequence $((f^n)')$ is also normal and converges uniformly towards $0$. In particular, it does not take the value $\lambda$ in $U \cap J_\varepsilon^c$ for $n$ large enough. Finally, assume $U$ is the basin of attraction of $\infty$. As no critical point of $f$ lies in $U$, then the same result follows from the fact that $(|(f^n)'|)$ converges uniformly to $\infty$ in  $U \setminus J_\varepsilon$.
\end{proof}
We now come to the proof of Theorem\ref{equi_poly2}.
\begin{proof}[Proof of Theorem~\ref{equi_poly2}]
Consider the logarithmic potential $u_n$ of $\nu_n^\lambda$ defined by :
\[u_n:=  \frac{1}{d^n-1} \log |(f^n)'-\lambda| -\frac{1}{d^n-1} \log d^n\]
where the additive constant $\frac{1}{d^n-1} \log d^n$ is chosen to take into account the fact that $(f^n)'$ is not unitary since the coefficient of its dominating term is $d^n$.
 We want to apply \cite[Lemma 15.5]{Brolin} which, in our case, say that if the limit inferior of the $u_n$ is non-positive on $J_f$ then $\mu'=\mu_f$, ending the proof.
Observe that $|f'|$ is uniformly bounded in the compact set $J_f$ (say by a constant $M$) so that $|(f^n)'|\leq M^n$ by the chain rule, since $J_f$ is a $f$-invariant set. In particular:
\[ \log |(f^n)'-\lambda|  \leq \log (|(f^n)'| + |\lambda|) \leq \log (M^n + |\lambda|) \]
on $J_f$. Hence, the logarithmic potential $u_n$ of $\nu_n^\lambda$  satisfies
\[u_n \leq (d^{-n}\log (M^n + |\lambda|) -\frac{1}{d^n-1} \log d^n \ \text{ on }J_f~,\]
which goes to $0$ with $n$. The result follows.
\end{proof}

\begin{remark} \normalfont
\begin{enumerate}
	\item Using potential theory is very efficient. Observe that we can prove, using the same proof, the equidistribution of the measures $\nu^{\lambda,k}_n$ defined by:
	\[ \nu^{\lambda,k}_n := \frac{1}{d^n-k}\sum_{(f^n)^{(k)}(z)-\lambda} \delta_z         \]
	for any $\lambda \neq 0$. 
	\item Nevertheless, the method works only for polynomials (Brolin method already fails for rational maps in the case of the equidistribution of preimages of a point). Moreover, it can be easily adapted to the setting of bifurcations only in special cases, as the quadratic family (see e.g.~\S\ref{sec:quad}).
\end{enumerate}
\end{remark}

\subsection{Complement 2: Lyubich's Inverse branches approach}
In this section, we prove the following theorem:
\begin{theorem}\label{equi_poly3}
Assume that $f$ is a hyperbolic polynomial of degree $d\geq2$. Then, for all $\lambda \neq 0$, we have the equidistribution:
\[ \lim_{n\to \infty} \nu^{\lambda}_n  = \mu_f ,\]
in the sense of measures. 
\end{theorem}
\begin{proof}
We can assume that no critical point is in the exceptional set : the preimages of any critical point accumulate toward the equilibrium measure (if not $f =z^d$ and the result follows from easy direct computations).
We first assume that no critical points is sent to another after some iterates and then explain what modifications need to be done to get to the general case. In particular, no critical point is periodic.  With our assumptions, the Julia set is uniformly expanding: choosing some smooth metric in a neighborhood $J_\eta$ of $J_f$, there exists $\rho >1$ such that the image of a ball of radius $\delta$ in  $J_\eta$ contains a ball of radius $\rho \delta$.

Fix $\varepsilon>0$ and $k\in \N$ large enough. Take some small ball $B_c$ around each critical points $c$, we consider the preimages $f^{-m}(B_c)$ for all $c$ and all $m\in \N$. Choosing  $B_c$ small enough guarantees that each $f^{-m}(B_c)$ consists of $d^m$ distinct connected components that we denote $B_c^{i,m}$ such that:
\begin{itemize}
\item the diameters of $B_c^{i,m}$ goes to $0$ when $m \to \infty$ by expansivity  ;
\item for $m \geq m_0$, $B_c^{i,m} \subset  J_\eta$ (all the preimages of the critical points end up in $J_\eta$);
\item for $(c,i,m)\neq (c',i',m')$, we have $B_c^{i,m} \cap B_{c'}^{i',m'} =\varnothing $ (this is clear for $m$ small enough by restricting the $B_c$ and it follows from the expansivity when all the preimages are in  $J_\eta$).
\end{itemize} 

Let $a:= \sup_{x\in \partial \left\{ \cup_{j=0}^{k-1} f^{-j}(B_c) \right\} } |(f^{k})'(x)| >0$ (no critical point lies in that set) and choose $M$ so that $M\cdot a \geq 2 |\lambda|$. For $m'$ large enough, we have that for all $z\in B_c^{i,m}$,  $|(f^n)'(z)| \geq M$ for all $m>n \geq m'$ (this follows from the expansivity of $J_\eta$). Take $m \geq m' +k  $ and let $ m > n \geq m-k$. Then for $x \in \partial B_c^{i,n}$, we have by the chain rule and the above:
 \[ |(f^m)'(x)|= |(f^{n})'(x)|\cdot|(f^{m-n})'(f^{n}(x))| \geq M\cdot a \geq 2 |\lambda|~.\]
On the other hand, if $c_{i,n}$ denotes the preimage by $f^n$ of $c$ that belongs to $B_c^{i,n}$, then:
$$ (f^m)'(c_{i,n})=0.$$
By the mean value theorem, we deduce that there exists a point $x_{c,i,n} \in B_c^{i,n}$ such that $(f^m)'(x_{c,i,n})=0$. Summing over all $i$ times the number of critical points gives $(d-1) d^n$ such points. Summing over all $n$ gives $d^m-d^{m-k}$ such points. In particular, we have the decomposition:
\begin{equation*}
\nu^\lambda_n  = \frac{1}{d^m-1} \sum_c \sum_{m-k\leq n<m} \sum_{i} \delta_{x_{c,i,n}}  +\theta_n  ~,
 \end{equation*}
 where $\theta_n$ is a measure of mass $\frac{d^{m-k}}{d^m-1} $ which can be taken $<\varepsilon$ for $k$ large enough. 
Denote:
\[\frac{1}{d^n}\sum_{i} \delta_{x_{c,i,n}}:= \mu_{c,n}~.\]
Then $\mu_{c,n} $ is a probability measure which is known to converge to $\mu$ when $n \to \infty$. Indeed, the sequence of measures equidistributed on the preimages of $c$ converges to $\mu$ when $m \to \infty$ and $|x_{c,i,n}-c_{i,n}| \to 0$ uniformly in $i$. Combining the above we have:
\[ \nu^\lambda_n  = \frac{1}{d^m-1} \sum_c \sum_{m-k\leq n<m}   \frac{d^n}{d^m-1} \mu_{c,n}  +\theta_n \to \mu_f~,\] 
in the sense of measures.

Let us now explain what are the modifications in the case where one critical point is sent to another after some iterates (assume that all the critical points are simple): given $m_0$ large enough, using Lyubich inverse branches ideas (see e.g. \cite{ljubich,briendduval2}) we can construct $(1-\varepsilon)d^{m_0}$ preimages of some small disks centered around each of the critical points. Removing some of those preimages, we can assume that $(1-2\varepsilon)d^{m_0}$ of those preimages lie in the expanding neighborhood $J_\eta$ of $J_f$. However for two of such preimages, $B$ and $B'$ may satisfy $f^j(B)\cap B' \neq 0$, this will happen if and only if $f^j(c)=c'$ where $c$ and $c'$ are the critical points whose preimages by $f^j$ are in $B$ and $B'$. Whenever that happen, we remove $B'$ from the list of preimages that we keep. Now, we can construct for each of those disks exactly $d^m$ preimages by $f^m$ for all $m$ (in particular, we take $m \gg j$ for all $j$ as above). Now, when counting as above the points $x$ in the preimages (more precisely their image by $f^l$ for $k\geq l>0$) such that $(f^m)'(x)=0$, observe that the preimages such that $f^j(B)\cap B'$ give several such points $x$ in their preimages because of the multiplicity of $(f^m)'(c_{i,n})=0$ ($c_{i,n}$ is the preimage of the critical point $c$ which is also a preimage of the point $c'$). The rest of the proof is similar.

Finally, if some critical points has multiplicity, the argument is the same taking into account that in each preimages, one has to take several $x$ counting the multiplicity.
\end{proof}

\begin{remark}\label{remark_equidistibution2}\normalfont
\begin{enumerate}
	\item The proof extends easily to the case of rational maps, with the same hypothesis on the critical points and hyperbolicity of the Julia set. It has also the interest to show what is the "typical" behavior of a point $x$ such that $(f^n)'(x)=\lambda$): the orbits stay near the Julia set for a long time, "gaining" hyperbolicity, then it is "ejected" and goes very close to a critical point, losing that hyperbolicity then its orbit follows the orbit of the critical point for a small time. Notice that the proof could be improved to work with weakened hypothesis using more subtlety inverse branches ideas and Pesin theory. 
	\item Nevertheless, the method of inverse branch cannot be used in the setting of bifurcations (or H\'enon maps).  
\end{enumerate}
\end{remark}

\section{In the parameter space of polynomial maps}

We consider now the following polynomial map in $\mathbb{C}^{d-1} \times \C$ defined in the introduction :
\begin{eqnarray*}
f:  \C^{d-1} \times \C  & \longrightarrow & \C^{d-1} \times \C  \\
	  ((c,a),z )   &   \longmapsto & ((c,a), P_{c,a}(z) ) .  
\end{eqnarray*}
Our aim in this section is to prove Theorem~\ref{equi_bif}.

\subsection{A partial tangent map}
We consider the "partial" tangent map: 
 \begin{eqnarray*}
\tilde{f} : & {\C}^{d-1} \times \C \times \C   \longrightarrow & {\C}^{d-1} \times \C \times \C  \\
	          & ((c,a),z,u )      \longmapsto &    ((c,a), P_{c,a}(z), P'_{c,a}(z)\cdot u )   
\end{eqnarray*}
We shall still denote by $\tilde{f}$ its homogenous extension to $\P^{d+1}(\C) \times \P^1(\C)$:
 \begin{eqnarray*}
\tilde{f}:  \P(\C)^{d} \times \P^1(\C)  & \longrightarrow & \P(\C)^{d} \times \P^1(\C) \\
	         ([c:a:z:t],[u:v] )   &   \longmapsto & \left(\left[ct^{d-1}:a t^{d-1}:\frac{1}{d}z^d+\sum_{j=2}^{d-1}(-1)^{d-j}\frac{\sigma_{d-j}(c)}{j}z^j+a^d:t^d\right],\right. \\
	         &     & \hspace*{1cm}  \left.\left[z^{d-1}+\sum_{j=2}^{d-1}(-1)^{d-j}\sigma_{d-j}(c)z^{j-1})u:t^{d-1}v\right]\right).   
\end{eqnarray*}

In homogenous coordinate, with the convention $c_0=0$, the indeterminacy set of $\tilde{f}$ is
\begin{align*}
 I(\tilde{f})=&   \{t=0\}\cap \left\{\frac{1}{d}z^d+\sum_{j=2}^{d-1}(-1)^{d-j}\frac{\sigma_{d-j}(c)}{j}z^j+a^d=0\right\} \\
               &  \cup  \{t=0\}\cap\{u=0\} \cup\bigcup_{i=0}^{d-2} \{z=c_i\}\cap \{v=0\}.
\end{align*}
Observe that the (invariant) set $\{[c:a:z:0]\}\times \P^1\backslash  I(\tilde{f})$ is sent by $\tilde{f}$ to the fixed point $I':=([0:0:1:0],[1:0])$. For the rest of the text, we let $W$ be a small neighborhood of $I'$ such that $\tilde{f}(W)\subset W$. Observe that one cannot have $\tilde{f}(W)\Subset W$ because the fixed point $I'$ has neutral directions given by the parameters space variables.

 Let $\omega_1$ (resp. $\omega_2$) be the pull back of the (normalized) Fubini-Study form $\omega_{\P^{d}}$ of $\P^{d}$ (resp. $\omega_{\P^1}$ on $\P^{1}$) by the canonical projection to the first (resp. second) factor  $P_1: \P^{d}\times \P^1 \to \P^{d}$ (resp. $P_2: \P^{d}\times \P^1 \to \P^1$ ). Let $\{\omega_i\}$ denote the class of $\omega_i$ in $H^{1,1}( \P^{d}\times \P^1)$. 
\begin{lemma}
The map $\tilde{f}$ is algebraically stable i.e. $(\tilde{f}^*)^n=(\tilde{f}^n)^*$ in $H^{1,1}( \P^{d}\times \P^1)$. The action of $(\tilde{f})^*$ is given in the $(\{\omega_1\}, \{\omega_2\})$ basis by the matrix:
\[\begin{pmatrix} d & d-1 \\
                    0  & 1 
\end{pmatrix}~.\]
\end{lemma}
\begin{proof} 
We again rely on the characterization of the algebraic stability by the existence of some hypersurface sent to the indeterminacy set by some iterate (see~\cite{Sibony}).  Let $A$ be such an hypersurface, then $P_1(A)$ is an algebraic subvariety of $\P^{d}$. Assume  $P_1(A)=\P^{d}$. Let $[c:a:z:t]$ be a fixed point of $f$ such that $f'(z) \neq 0$ (it is obvious that there are infinitely many such points), and let $[u:v] \in \P^1$ such that $([c:a:z:t],[u:v])\in A$. By the chain rule, $\tilde{f}^n([c:a:z:t],[u:v])=([c:a:z:t],(f^n)'(z)\cdot u) \notin I(\tilde{f})$. This is a contradiction.

In particular,  $P_1(A)\neq \P^{d}$ so that $A = H\times \P^1$. 
We extend $f$ as a rational mapping to $\P^d$ by its homogenous extension to $\P^{d}$ and still denote it by $f$. Observe first that the map $f$ is algebraically stable: indeed, if it is not the case, since $f$ is a polynomial map in $\C^{d}$, the only possibility is that $\{t=0\}$ is sent to the indeterminacy set. This is not the case since $f (\{t=0\}) \backslash I(f) =[0:1:0]$.

If $H \neq \{t=0\}$, let $p \in H \cap \C^{d}$.
 By the chain rule, $\tilde{f}^n$ is the extension of $\tilde{f}^n(p,u)=(f^n(p),(f^n)'(p)\cdot u) $. Observe that $\{p\}\times \{[0:1]\} \subset \{p\}\times \P^1$ is sent to $\{f(p)\}\times\{[0:1]\}$ whether $p$ is critical or not. It follows from this that the line $\{p\}\times \P^1$ is not sent to the indeterminacy set by $ \tilde{f}^n$. 
 We assume now that  $H = \{t=0\}$. As the point $I'$ is sent to itself by $\tilde{f}$ then $\tilde{f}^n(I')=I' \notin I(\tilde{f})$. So $\tilde{f}$ is algebraically stable. The rest of the proof follows. 
 \end{proof}

\subsection{Convergence towards the Green current of the tangent map}

In the sequel, we let $\mathcal{T}$ be the Green current of the map $f$.
One result which follows applying the same proof as that of Lemma~\ref{lm:cvGreen}is the following.

\begin{lemma}\label{lm:cvGreenpara}
The sequence of positive closed currents $d^{-n} (\tilde{f}^n)^*(\omega_2)$ converges to the Green current $T_{\tilde{f}}= P_1^*(\mathcal{T})$ of $\tilde{f}$.
\end{lemma}
Unfortunately, this result itself is not useful in our case. Indeed, to define the activity current $T_i$ one needs to \emph{intersect} the Green current $T_{\tilde{f}}$ with the current of integration on $\{z=P_{c,a}(c_i)\}$. Such an operation is \emph{not} continuous in the sense of currents, hence Lemma~\ref{lm:cvGreenpara} does not provide sufficient informations for our purpose. Instead, we shall prove directly the following:
\begin{proposition}
The sequence $(d^{-n} (\tilde{f}^n)^*(\omega_2) \wedge [z= P_{c,a}(c_i)])_n$ of positive closed currents of bidegree $(2,2)$ is well defined and converges to the current
\[\mathcal{T}_{\tilde{f}}:= P_1^*(\mathcal{T}) \wedge [z= P_{c,a}(c_i)]~.\]
\end{proposition}
\begin{proof}
Observe first that  $I(\tilde{f}^n) \cap \{z= P_{c,a}(c_i)\}$ has codimension $3$ (this follows from the fact that $ \{z= P_{c,a}(c_i)\} \cap  \{P_{c,a}^{n-1}(z)= c_i\}$ has codimension $2$ in $\C^{d} \times \C$. In particular, $d^{-n} (\tilde{f}^n)^*(\omega_2) \wedge [z= P_{c,a}(c_i)]$ is well defined since the restriction of the current 
$d^{-n} (\tilde{f}^n)^*(\omega_2)$ to  $ [z= P_{c,a}(c_i)]$ is smooth outside $I(\tilde{f}^n)$.

 Write $\omega_2= \omega_1 + \omega_2-\omega_1$. Since $d^{-n} (\tilde{f}^n)^*(\omega_1)$ converges to $T_{\tilde{f}}$ with  local uniform convergence of the potentials,
all there is left to prove is that $d^{-n} (\tilde{f}^n)^*(\omega_2- \omega_1)\wedge [z= P_{c,a}(c_i)]$ converges to $0$ in the sense of currents. Observe that $\tilde{f}^*(\omega_2- \omega_1)$ is cohomologous to $\omega_2- \omega_1$ hence we can write 
\[\tilde{f}^*(\omega_2- \omega_1)= \omega_2- \omega_1+dd^c \varphi\]
where $\varphi$ is smooth outside $I(\tilde{f})$ and is a dsh function.  
A straight-forward induction gives:
\[d^{-n} (\tilde{f}^n)^*(\omega_2- \omega_1)= d^{-n}(\omega_2- \omega_1)+dd^c  \left(d^{-n}\sum_{k=0}^{n-1} \varphi \circ \tilde{f}^k \right) .\]
We write $\varphi_n:=d^{-n}\sum_{k=0}^{n-1}  \varphi \circ \tilde{f}^k$. We have reduced the problem of proving the convergence of $d^{-n} (\tilde{f}^n)^*(\omega_2) \wedge [z= P_{c,a}(c_i)]$ to  $\mathcal{T}_{\tilde{f}}= P_1^*(\mathcal{T}) \wedge [z= P_{c,a}(c_i)]$ to proving that the sequence $(\varphi_n)$ restricted to $\{z= P_{c,a}(c_i)\}$ converges to $0$ in $L^1(\{z= P_{c,a}(c_i)\})$. We denote by $\tilde{\varphi}_n$ the restriction of $\varphi_n$ to $\{z= P_{c,a}(c_i)\}$.

The sequence of functions $(\tilde{\varphi}_n)$ is then a sequence of dsh functions on $\{z=P_{c,a}(c_i)\}$ by construction. Furthermore, if $U$ denotes a stability component of $c_i$ on which its orbit is bounded, then for all $V \Subset U$, the families $(P^n_{c,a}(c_i))$ and $((P^n_{c,a})'(c_i))$ are equicontinuous hence equibounded. In particular, for any compact set $K'\in \C$, there exist compact sets $L\subset \C$ and $K \subset \C$ such that for all $(c,a) \in V$ and all $u \in K'$, for any $n \in \N^*$, one has:
\[P^n_{c,a}(P_{c,a}(c_i)) \in L \ \mathrm{and} \ ((P^n_{c,a})'(P_{c,a}(c_i))\cdot u)\in K'.\]

Let $\nu_{L\times K}$ be a smooth (hence PB) probability measure with support in $L\times K$. Since $\varphi$ is smooth in $\C^{d} \times \C$, we have
$\|\tilde{\varphi}_n\|_{\infty, V}\leq nd^{-n} \|\varphi\|_{\infty, L\times K}\leq C$ where $C$ is a constant that does not depend on $n$.  
In particular, $\|\varphi_n\|_{\nu_{L\times K}}$ is uniformly bounded. As all the $DSH$-norms are equivalent, we deduce that $(\tilde{\varphi}_n)$ is bounded in $L^1$ for the standard Fubini Study measure in $\{z= P_{c,a}(c_i)\}$. 

On the other hand, the sequence $d^{-n} (\tilde{f}^n)^*(\omega_2)\wedge [z= P_{c,a}(c_i)]$ is bounded in mass hence we can extract a converging subsequence. Its limit is a positive closed current cohomologuous to $\omega_1 \wedge [z= P_{c,a}(c_i)]$ hence it can be written as $\tilde{\pi}^*(T')$ where $T'$ is some $(1,1)$ current in $[z= P_{c,a}(c_i)]$ and $\tilde\pi$ denote the projection to the first coordinate of $\{z= P_{c,a}(c_i) \}$.
 Extracting again, we can assume that $(\tilde{\varphi}_n)$ converges in $L^1$ (it is bounded in DSH ). Its limit $V$ satisfies $dd^c V= \tilde{\pi}^*(T'-\mathcal{T})$ hence it is constant on each fiber of $\tilde{\pi}$. In other words, $V=\pi_1^*(v)$ for some dsh function $v$ on $\{(c,a,z)\, ; \ z= P_{c,a}(c_i) \}$. The above argument shows that $v=0$ on any stability component of $c_i$ on which its orbit is bounded. The case of the component on which $c_i$ goes to $\infty$ is similar using this time a arbitrary compact sets of the form $L \times K$ where $K$ is a compact set on the complementary of $0 \in \C$ (indeed, the derivative converges to $\infty$). 

It follows that the function $v$ is equal to $0$ in the stability locus. Now, the unstability locus has empty interior (this is classical and follows directly from Montel theorem) and a dsh function that is $0$ outside such set is identically $0$ by pluri-fine continuity. It follows that $\tilde{\pi}^*(T')= \tilde{\pi}^*(\mathcal{T})$ which ends the proof.
\end{proof}
The following is now obvious:
\begin{corollary}\label{cor:cvpara}
For any PB probability measure $\nu$ on $\P^1$, the sequence of currents $d^{-n} (\tilde{f}^n)^*(\pi_2^*(\nu))\wedge [z= P_{c,a}(c_i)]$ converges to $\mathcal{T}_{\tilde{f}}$ in the sense of currents.
\end{corollary}

\subsection{Value distribution of the derivative}

For $u_0 \in \C$, recall that we denote by $[u=u_0]$ the current of integration on the hyperplane $\P^d\times \{u_0\}$.
\begin{proposition}

For any $u_0 \in \C^*$, the sequence of currents $d^{-n} (\tilde{f}^n)^*([u=u_0])$ converges to $\mathcal{T}_{\tilde{f}}\wedge [z= P_{c,a}(c_i)]$ in the sense of currents.
\end{proposition}
\begin{proof} By the above Corollary~\ref{cor:cvpara}, one deduce that outside a polar set of $u_0 \in \C$, $d^{-n} (\tilde{f}^n)^*([u=u_0])\wedge [z= P_{c,a}(c_i)]$ converges to $\mathcal{T}_{\tilde{f}}$ in the sense of currents. Choose such a $u_0 \neq 0$. For $\lambda \in \C^*$, consider the map $D_\lambda: (c,a,z,u) \mapsto (c,a,z, \lambda u)$. Observe that 
$$d^{-n} (\tilde{f}^n)^* ( [u=\frac {1}{\lambda} u_0]) \wedge [z= P_{c,a}(c_i)] = D_\lambda^*(d^{-n} (\tilde{f}^n)^* ( [u=u_0]) \wedge [z= P_{c,a}(c_i)]).$$ 
Now, $D_\lambda^*(d^{-n} (\tilde{f}^n)^* ( [u=u_0]) \wedge [z= P_{c,a}(c_i)])$ converges to $D_\lambda^*(\mathcal{T}_{\tilde{f}})$ in the sense of currents by assumption and continuity of $D_\lambda^*$. Since $\mathcal{T}_{\tilde{f}}= P_1^*(\mathcal{T} \wedge [z= P_{c,a}(c_i)] )  $, we have $D_\lambda^*(\mathcal{T}_{\tilde{f}})=\mathcal{T}_{\tilde{f}}$. The result follows since taking a $u_1 \in \C^*$ arbitrary, we can write $ \lambda u_1 = u_0$ for a suitable $\lambda$.
\end{proof}

We can now prove Theorem~\ref{equi_bif}. Let us restate it:
\begin{theorem}\label{equi_bif1}
Let $\nu$ be a PB measure in $\C$, then we have the equidistribution:
\[\lim_{n\to \infty} \frac{1}{d^n}\int_\C[(P_{c,a}^n)'(P_{c,a}(c_i))=\lambda] \mathrm{d}\nu(\lambda) = T_i~.\]
In particular, outside a polar set of $\lambda \in \C$, we have:
\[ \lim_{n\to \infty} \frac{1}{d^n} \left[ (P_{c,a}^n)'(P_{c,a}(c_i)) =\lambda \right] = T_i.\] 
\end{theorem}

\begin{proof}
Denote by $p:\C^d\rightarrow\C^{d-1}$ the projection onto the parameter variables, i.e. the map defined by $p(c,a,z)=(c,a)$.
Let $u_0 \in \C^*$, then observe that $(\tilde{f}^n)^*([u=u_0])  \wedge [z= P_{c,a}(c_i)] \wedge [u=1] = (\tilde{f}^n)^*([u=1])  \wedge [z= P_{c,a}(c_i)]\wedge [u=u_0^{-1}]$. Let $\mathrm{inv}:\P^1\to \P^1$ be the rational map $z \mapsto  z^{-1}$. In other words, the currents $(\tilde{f}^n)^*([u=u_0] \wedge [z= P_{c,a}(c_i)]) \wedge [u=1]$ and $(\tilde{f}^n)^*([u=1] \wedge [z= P_{c,a}(c_i)]) \wedge  P_2^* (\mathrm{inv}^*([u=u_0]))$ are equal.

By Fubini Theorem, we have for $\nu$,
\[  (\tilde{f}^n)^*(P_2^*(\nu)) \wedge [z= P_{c,a}(c_i)]\wedge [u=1] = (\tilde{f}^n)^*([u=1]) \wedge [z= P_{c,a}(c_i)] \wedge P_2^*( \mathrm{inv}^*(\nu)). \]   
As before, it is enough to prove the first point for any PB measure, so we take $\nu=\omega_2$ the Fubini-Study form ($\mathrm{inv}^*(\omega_2)=\omega_2$). 
By the above proposition  $d^{-n} (\tilde{f}^n)^*([u=1]) \to \mathcal{T}_{\tilde{f}}$ in the sense of currents, hence  $d^{-n} (\tilde{f}^n)^*([u=u_0]) \wedge [z= P_{c,a}(c_i)]\wedge P_2^*( \omega_2)\to \mathcal{T}_{\tilde{f}}\wedge P_2^*( \omega_2)$ since $P_2^*(\omega_2)$ has continuous potentials. By continuity of $(P_1)_*$ and $p_*$, and using the equality between the currents $(\tilde{f}^n)^*([u=u_0] \wedge [z= P_{c,a}(c_i)]) \wedge [u=1]$ and $(\tilde{f}^n)^*([u=1] \wedge [z= P_{c,a}(c_i)]) \wedge  P_2^* (\mathrm{inv}^*([u=u_0]))$ gives the first part of the theorem.

The second part follows since the set of $\lambda$ for which the convergence
\[ \lim_{n\to \infty} \frac{1}{d^n} \left[ (P_{c,a}^n)'(P_{c,a}(c_i)) =\lambda \right] = T_i.\] 
does not hold has zero measure for all the PB measures of $\C$. Such a set is polar.
\end{proof}

\subsection{In the quadratic family}\label{sec:quad}
We now focus on the quadratic family parametrized by $p_c(z):=z^2+c$, for $c\in\C$. 
We denote by $K_c$ the filled-in Julia set of $p_c$ and by $J_c$ the Julia set:
\[K_c:=\left\{z\in \C~|~\left(p_c^{\circ n}(z)\right)_{n\in \N}\text{ is bounded}\right\}\quad\text{and}\quad 
J_c:=\partial K_c.\]
The Mandelbrot set $M$ is the set of parameters $c\in \C$ such that $0\in K_c$. It is a compact connected subset of $\C$. The Green function $g_M:\C\to [0,+\infty[$ of the Mandelbrot set is
\[g_M(c)=g_{p_c}(c)~, \ c\in\C~.\]
The  bifurcation measure of the quadratic family is $\mu_\bif=\frac{1}{2}dd^cg_M$. Its support is the boundary of the Mandelbrot set $M$. Moreover, the probability measure $\mu_M=2\mu_\bif$ is the harmonic measure of the Mandelbrot set (see e.g.~\cite{carleson}). 

Pick a polynomial $\lambda(c)\in\C[c]$ and for any integer $n\geq1$, let
\[\nu_n^\lambda:=\frac{1}{2^n}\sum_{(p_c^n)'(c)=\lambda(c)}\delta_c~,\]
where we take into account the multiplicity in the sum. Beware that given a polynomial $\lambda$, the measure $\nu_n^\lambda$ has mass $(2^{n}-1)/2^n\sim1$ for $n$ large enough.
We want here to prove the following more general result.

\begin{theorem}\label{equiMand}
For any polynomial $\lambda(c)\in\C[c]$, the sequence of finite measures $(\nu_n^\lambda)_{n\geq1}$ converges weakly to the harmonic measure $\mu_\Mand$ of the Mandelbrot set.
\end{theorem}

In~\cite[Lemma 2]{multipliers}, Buff and the first author prove the following we rely on:
\begin{lemma}[Buff-Gauthier]\label{lemma:comparisonM}
Any subharmonic function $u:\C\to [-\infty,+\infty[$ which coincides with $g_M$ on $\C\setminus M$ coincides with $g_M$ on $\C$. 
\end{lemma}

This lemma is similar to the extremality property used in Brolin's approach in Section~\ref{sec:brolin}.

\begin{proof}[Proof of Theorem~\ref{equiMand}]
Pick any polynomial $\lambda(c)\in\C[c]$ and for $n\geq1$, let
\[u_n^\lambda(c):=\frac{1}{2^n}\log|(p_c^n)'(c)-\lambda(c)|~, \ c\in\C~,\]
so that $\nu_n^\lambda=dd^c(u_n^\lambda)$. By definition, the sequence of subharmonic functions $(u_n^\lambda)$ is locally uniformly bounded from above. Hence, by Hartogs lemma,
\begin{itemize}
\item either $u_n^\lambda\to-\infty$ uniformly locally on $\C$,
\item or it admits subsequences which converge in $L^1_{\textup{loc}}(\C)$.
\end{itemize}
According to Lemma~\ref{lemma:comparisonM}, to conclude it is sufficient to prove that $u_n^\lambda\to g_M$ pointwise on $\C\setminus M$. Indeed, if this holds true any $L^1_\textup{loc}$ limit $u$ of $(u_n^\lambda)$ is subharmonic on $\C$ and coincides with $g_M$ on $\C\setminus M$. By Lemma~\ref{lemma:comparisonM}, this means that $u_n^\lambda\to g_M$ in $L^1_\textup{loc}$ and the conclusion follows.

Pick now $c\in\C\setminus M$, then there exists $C>0$ universal such that
\[\left|\frac{1}{2^n}\log|p_c^n(c)|-g_M(c)\right|\leq C\frac{n}{2^n}~,\]
for all $n\geq1$ (see~\cite[Lemma 4.1]{distribGV}). On the other hand, by the chain rule, we have $(p_c^n)'(c)=2^n\prod_{k=0}^{n-1}p_c^k(c)$, hence 
\begin{eqnarray*}
\left|\frac{1}{2^n}\log|(p_c^n)'(c)|-g_M(c)\right| & \leq & \sum_{k=0}^{n-1}\frac{1}{2^{n-k}}\left|\frac{1}{2^k}\log|p_c^k(c)|-g_M(c)\right|+\frac{1}{2^n}g_M(c)+\frac{n}{2^n}\log2\\
& \leq & \frac{1}{2^n}\left(g_M(c)+n\log 2+Cn^2\right)~.
\end{eqnarray*}
In particular, we get $\lim_{n\to\infty}2^{-n}\log|(p_c^n)'(c)|=g_M(c)$ and $(p_c^n)'(c)\to\infty$ as $n\to\infty$. To conclude, we just have to remark that
\begin{eqnarray*}
\lim_{n\to\infty}u_n^\lambda(c) & = & \lim_{n\to\infty}\frac{1}{2^n}\log|(p_c^n)'(c)|+ \lim_{n\to\infty}\frac{1}{2^n}\log\left|1-\frac{\lambda(c)}{(p_c^n)'(c)}\right|\\
 & = & g_M(c)~,
\end{eqnarray*}
which concludes the proof of the Theorem. 
\end{proof}

The proof we implement here can not be easily generalized to higher degrees to prove that exceptional sets from Theorem~\ref{equi_bif} are empty. Indeed, the pointwise estimates in $\C\setminus M$ used in the proof we need have to be replaced with more elaborate estimates, as in \cite{distribGV}. The estimates we can prove in that context concern only the \emph{fastest} escaping critical point.

Moreover, this strategy is specifically designed for proving codimension $1$ equidistribution phenomena, since the proof gives the $L^1_{\textup{loc}}$ convergence of the potentials. On the other hand, the proof of Theorem~\ref{equi_bif} we give above can be generalized to higher codimension objects.

\bibliographystyle{short}
\bibliography{biblio}
\end{document}